\documentclass[12pt]{amsart}
\usepackage{amssymb}
\usepackage{hyperref}

\theoremstyle{plain}
\newtheorem{theorem}{Theorem}[section]

\newtheorem{lemma}[theorem]{Lemma}

\theoremstyle{definition}

\newtheorem{remark}[theorem]{Remark}

\theoremstyle{remark}

\numberwithin{equation}{section}

\newcommand{\N}{\mathbb N}
\newcommand{\Z}{\mathbb Z}

\newcommand{\R}{\mathbb R}
\newcommand{\C}{\mathbb C}

\newcommand{\GL}{\operatorname{GL}}

\newcommand{\Ot}{\operatorname{O}}
\newcommand{\SO}{\operatorname{SO}}
\newcommand{\SU}{\operatorname{SU}}
\newcommand{\U}{\operatorname{U}}
\newcommand{\Sp}{\operatorname{Sp}}

\newcommand{\so}{\mathfrak{so}}

\newcommand{\op}{\operatorname}

\newcommand{\Hom}{\operatorname{Hom}}

\newcommand{\tr}{\operatorname{tr}}
\newcommand{\diag}{\operatorname{diag}}

\newcommand{\mi}{\mathtt{i}}

\newcommand{\intervalo}{[\![1,n]\!]}

\title[Branching multiplicity space]{On the $\textrm{SO}(n+3)$ to $\textrm{SO}(n)$ branching multiplicity space}
\author{Emilio~A.~Lauret, Fiorela Rossi Bertone}
\address{L.: Institut f\"ur Mathematik, Humboldt Universit\"at zu Berlin, Unter den Linden 6, 10099 Berlin, Germany.
 Permanent affiliation: CIEM--FaMAF (CONICET), Universidad Nacional de C\'ordoba, Medina Allende s/n, Ciudad Universitaria, 5000 C\'ordoba, Argentina.}
\email{elauret@famaf.unc.edu.ar}
\address{R.B.: CIEM--FaMAF (CONICET), Universidad Nacional de C\'ordoba, Medina Allende s/n, Ciudad Universitaria, 5000 C\'ordoba, Argentina.}
\email{rossib@famaf.unc.edu.ar}
\thanks{This research was partially supported by grants from CONICET and FONCyT. The first named author was supported by the Alexander von Humboldt Foundation}
\date{\today}

\begin{document}

\begin{abstract}
We study the decomposition as an $\textrm{SO}(3)$-module of the multiplicity space corresponding to the branching from $\textrm{SO}(n+3)$ to $\textrm{SO}(n)$. 
Here, $\textrm{SO}(n)$ (resp.\ $\textrm{SO}(3)$) is considered embedded in $\textrm{SO}(n+3)$ in the upper left-hand block (resp.\ lower right-hand block).
We show that when the highest weight of the irreducible representation of $\textrm{SO}(n)$ interlaces the highest weight of the irreducible representation of $\textrm{SO}(n+3)$, then the multiplicity space decomposes as a tensor product of $\lfloor (n+2)/2\rfloor$ reducible representations of $\textrm{SO}(3)$. 
\end{abstract}

\maketitle


\section{Introduction} \label{sec:intro}
The branching law from a compact Lie group $G$ to a closed subgroup $K$ describes how an irreducible representation $\pi$ of $G$ decomposes when is restricted to $K$ (see \cite[Ch.~IX]{Knapp-book-beyond} and   \cite[Ch.~8]{GoodmanWallach-book-Springer} for comprehensive texts, and \cite[\S1]{Knapp01} for a detailed historical review). 
Since 
\begin{equation}
\pi\simeq \bigoplus_{\tau\in\widehat K} \, \tau\otimes \Hom_K(\tau,\pi)
\end{equation}
as $K$-modules ($\widehat K$ denotes the unitary dual of $K$, and $K$ acts on the right-hand side at the left in each term), the branching law is determined by the dimension of the \emph{branching multiplicity space} (or just \emph{multiplicity space})
$\Hom_K(\tau,\pi),$ for each $\tau \in \widehat K$.
In other words, $\dim \Hom_K(\tau,\pi)$ is the number of times that $\tau$ occurs in $\pi|_K$.

For $d\geq d'\geq1$, let $G$, $K$ and $H$ be given by a row in the table
\begin{equation}\label{eq:Grassmanianos}
\begin{array}{c@{\hspace{10mm}}c@{\hspace{10mm}}c@{\hspace{10mm}}c}
G & K & H &\text{type} \\ \hline 
\rule{0pt}{14pt}
\SO(d+d') & \SO(d) & \SO(d') &\text{orthogonal}\\
\U(d+d') & \U(d) & \U(d')&\text{unitary} \\
\Sp(d+d') & \Sp(d) & \Sp(d') &\text{symplectic}\\
\end{array}.
\end{equation}
In the sequel, we assume $K$ (resp.\ $H$) embedded in $G$ in the upper left-hand block (resp.\ lower right-hand block).
The quotient $G/(K\times H)$ is a compact symmetric space called (real, complex or quaternionic) Grassmannian space.

We now examine some consequences from the fact that the subgroups $K$ and $H$ commute to each other.
The subgroup of $G$ generated by $K$ and $H$ is isomorphic to $K\times H$; we denote it by $K\times H$. 
Thus, any irreducible representation of $K\times H$ is given by the outer tensor product $\sigma\otimes \tau$ for some $\sigma\in\widehat K$ and $\tau\in\widehat H$. 
Furthermore, the branching multiplicity space $\Hom_K(\sigma,\pi)$ carries the structure of an $H$-module. 
The action is given by $(k\cdot \varphi)(v) = \pi(k)\cdot\varphi(v)$, for $\varphi\in \Hom_K(\sigma,\pi)$ and $v$ in the underlying vector space $V_\sigma$ of $\sigma$. 
We conclude that the multiplicity of $\sigma\otimes\tau$ in $\pi|_K$ is equal to
\begin{equation}
\dim \Hom_{K\times H}({\sigma}\otimes {\tau},\pi) = \dim \Hom_H(\tau, \Hom_K(\sigma,\pi)).
\end{equation}
Therefore, an explicit decomposition as an $H$-module of the multiplicity space $\Hom_K(\sigma,\pi)$ may be seen as a more precise branching law from $G$ to $K\times H$.

The branching law from $G$ to $K\times H$ is known only for specific choices of $d$ and $d'$. 
We first review the case when $d'=1$. 
In the orthogonal case, $H\simeq \SO(1)=\{1\}$, thus the problem reduces to the classical branching from $\SO(d+1)$ to $\SO(d)$.
A similar situation takes place in the unitary case where $H\simeq \U(1)$. 
In conclusion, in both cases, under standard choices for a Cartan subalgebra and positiveness in the associated root systems (e.g.\ as in \cite{Knapp-book-beyond} or \cite{GoodmanWallach-book-Springer}), if $\lambda=(\lambda_1,\lambda_2,\dots)$ and $\mu = (\mu_1,\mu_2,\dots)$ are the highest weights of $\pi\in\widehat G$ and $\sigma\in \widehat K$ respectively, then $\sigma$ occurs in $\pi|_K$ if and only if $\mu$ \emph{simply interlaces} $\lambda$, which roughly speaking means 
\begin{equation}
\lambda_1\geq \mu_1\geq \lambda_{2}\geq \mu_2\geq \lambda_3\geq\cdots.
\end{equation}

The symplectic case still assuming $d'=1$, which was established by Lepowsky~\cite{Lepowsky71}, presents more difficulties than the previous cases. 
For example, it is not multiplicity-free. 
Moreover, the necessary condition $\Hom_K(\sigma,\pi)\neq0$ becomes in a \emph{doubly interlacing} $\lambda_i\geq \mu_i\geq \lambda_{i+2}$ for all $1\leq i\leq d$ (with $\lambda_{d+2}=\lambda_{d+3}=0$) between the coefficients $(\lambda_1,\dots, \lambda_{d+1})$ and $(\mu_1,\dots,\mu_d)$ of the highest weight of $\pi$ and $\sigma$ respectively. 
Wallach and Yacobi~\cite{WallachYacobi09} (see also \cite{Yacobi10} and \cite{KimYacobi12}) gave a clean decomposition for the multiplicity space.
They proved that
\begin{equation}\label{eq:neatWY}
\Hom_K(\sigma,\pi) \simeq \tau^{(1)}\otimes\dots\otimes \tau^{(d+1)}
\end{equation}
as $H$-modules, where $\tau^{(i)}$ denotes the irreducible representation of $H=\Sp(1)\simeq \SU(2)$ of dimension $s_i-t_i+1$, and $\{s_1\geq t_1\geq \dots \geq s_{d+1}\geq t_{d+1}\}$ is the decreasing rearrangement of $\{\lambda_1,\dots,\lambda_{d+1},\mu_1,\dots,\mu_{d},0\}$.
In their proof, they extended the Clebsch--Gordan formula to an arbitrary tensor product of irreducible representations of $\SU(2)$ (see Theorem~\ref{thm:ClebschGordan}).

We now consider $d'=2$. 
In the orthogonal case, $H=\SO(2)$ is abelian, so its representations are one-dimensional. 
In conclusion, a decomposition as an $H$-module of the multiplicity space does not provide more information than its dimension.
An implicit branching law from $G=\SO(d+2)$ to $K\times H=\SO(d)\times \SO(2)$ was given by Tsukamoto~\cite{Tsukamoto81}. 
The implicit term refers to the fact that the number of times that an irreducible representation $\sigma\otimes\tau_k$ of $K\times H$ appears in $\pi|_{K\times H}$ is given by the $k$-th coefficient of certain power series, where $\tau_k(h)$ acts on $\C$ by multiplication by $h^k$ for any $h\in S^1\simeq H$. 

Kim~\cite{Kim13} considered the unitary case when $d'=2$. 
He gave a decomposition of the branching multiplicity space as a $\U(2)$-module similar to \eqref{eq:neatWY}. 
In order to describe such decomposition, let $\pi\in\widehat G$ and $\sigma\in\widehat K$ with highest weights $\lambda=(\lambda_1,\dots,\lambda_{d+2})$ and $\mu=(\mu_1,\dots,\mu_d)$ respectively, and write $\{s_1\geq t_1\geq \dots \geq s_{d+1}\geq t_{d+1}\}$ for the decreasing rearrangement of $\{\lambda_1,\dots,\lambda_{d+2},\mu_1,\dots,\mu_{d}\}$.
Kim's decomposition (see \cite[Thm.~3.5]{Kim13}) is
\begin{equation}\label{eq:neatKim}
\Hom_K(\sigma,\pi) \simeq \C\otimes \tau^{(1)}\otimes\dots\otimes \tau^{(d+1)},
\end{equation}
where $\C$ is the one-dimensional representation given by $\det(h)^{\mu_1+\dots+\mu_d}$, and $\tau^{(i)}$ denotes the $(s_i-t_i+1)$-dimensional representation $\C\otimes \op{Sym}^{t_i-s_i}(\C^2)$ of $H$, with $h\in H$ acting on $\C$ by $\det(h)^{s_i}$ and $\C^2$ denotes the standard representation of $H$.

The next challenge is the case $d'=3$. 
The orthogonal case seems to be the simplest one since $\SO(3)\simeq \SU(2)/\{\pm1\}$ is three-dimensional, while $\U(3)$ and $\Sp(3)$ have dimensions $9$ and $21$ respectively.

The aim of this paper is to study, for $G=\SO(d+3)$, $K=\SO(d)$, $H=\SO(3)$, whether there is a clean decomposition of the $G$ to $K$ branching multiplicity space as an $H$-module as in \eqref{eq:neatWY} and \eqref{eq:neatKim}. 
The conclusion is that this decomposition is pretty dirty.

Nevertheless, when $\mu$ simply interlaces $\lambda$, we decompose the multiplicity space $\Hom_K(\sigma_\mu, \pi_\lambda)$ as a tensor product of $\lfloor \tfrac{d+2}2\rfloor$ (reducible) representations of $H$ (Theorems~\ref{thm2n:mainthm} and \ref{thm2n+1:mainthm}). 
Generically, each factor of the tensor product is non-trivial.
We recall that $\Hom_K(\sigma_\mu,\pi_\lambda)\neq0$ if and only if $\mu$ triply interlaces $\lambda$ (cf.\ last paragraph in \cite[\S IX.3]{Knapp-book-beyond}).

Furthermore, Theorems~\ref{thm2n:ending} and \ref{thm2n+1:ending} establish a partial decomposition of $\Hom_K(\sigma,\pi)$ as an $H$-module under certain coincidence among the coefficients of the highest weights of $\pi$ and $\sigma$. 
Moreover, the decomposition is reduced to a simple branching law from $\U(3)$ to $H=\SO(3)$ by using a duality by Knapp~\cite{Knapp01}.

Tsukamoto~\cite{Tsukamoto05} showed an implicit branching law from $G$ to $K\times H$. 
Similarly as mentioned above, his result gives the multiplicity of an irreducible representation $\sigma\otimes\tau$ of $K\times H$ in $\pi|_{K\times H}$ for $\pi\in\widehat G$ as the coefficient of certain power series. 
El Chami~\cite{Chami04}\cite{Chami12} applied the same method to the branching laws from $\SO(d+d')$ to $\SO(d)\times\SO(d')$ and from $\Sp(d+d')$ to $\Sp(d)\times \Sp(d')$.
In all cases, their main goal was to describe the spectra of the corresponding symmetric spaces.

The tools used in the proofs include Kostant's branching formula, Tsukamoto's implicit branching law from $\SO(d+3)$ to $\SO(d)\times \SO(3)$, and a duality of  Knapp~\cite{Knapp01} between the $\SO(3)$-representation $V_\pi^{\SO(d)}$ and certain canonical associated representation of $\U(3)$.  

The article is organized as follows. 
Section~\ref{sec:preliminaries} reviews standard facts used in the sequel. 
The case when $G=\SO(2n+3)$, $K=\SO(2n)$ and $H=\SO(3)$, called the type B case, is considered in Section~\ref{sec:casoB}. 
Similarly, Section~\ref{sec:casoD} deals with the type D case, that is, when $G=\SO(2n+4)$, $K=\SO(2n+1)$ and $H=\SO(3)$.

\section{Preliminaries} \label{sec:preliminaries}
In this section we introduce several tools used in the sequel, divided in subsections. 
Such tools are Kostant's branching formula, Wallach and Yacobi's extension of the Clebsch--Gordan rule, and standard facts on characters of compact groups and representations of $\U(3)$.

In what follows, we denote compact Lie groups by capital letters (e.g.\ $G$), their Lie algebras by the corresponding Gothic letter with the subscript $0$ (e.g.\ $\mathfrak g_0$), and their complexified Lie algebras by the corresponding Gothic letter (e.g. $\mathfrak g$). 

Furthermore, each time that a maximal torus $T$ is fixed in a compact Lie group $G$, therefore a Cartan subalgebra $\mathfrak t$ of $\mathfrak g$ is picked, we will use the following notation without any mention: 
$\Phi(\mathfrak g,\mathfrak t)$ denotes the associated root system, $W_{\mathfrak g}$ the Weyl group, and $P(G)$ the weight lattice of $G$.
Moreover, a positive system $\Phi^+(\mathfrak g,\mathfrak t)$ will be assumed, unless an order on $\mathfrak t^*$ is explicitly chosen. 
In any of these cases, we denote by $P^{++}(G)$ the set of $G$-integral dominant weights in $\mathfrak t^*$ and by $\rho_{\mathfrak g}$ half of the sum of the positive roots.

\subsection{Characters}
We first review standard facts for characters. 
We refer to \cite[Ch.~IV--V]{Knapp-book-beyond} for further details. 
Let $G$ be a compact connected semisimple Lie group, let $T$ be a maximal torus in $G$, and let $\pi$ be a finite-dimensional representation of $G$, that is, a continuous homomorphism $\pi:G\to\GL(V_\pi)$, where $V_\pi$ is the complex finite-dimensional underlying vector space of $\pi$. 

We denote by $\chi_\pi:G\to\C$ the character of $\pi$, that is, $\chi_\pi(g)=\tr(\pi(g))$. 
It is well known that $\chi_\pi$ determines $\pi$, that is, $\chi_\pi=\chi_{\pi'}$ if and only if $\pi$ and $\pi'$ are equivalent.  
It will be useful to consider $\chi_\pi$ as a formal power series $\sum_{\eta \in P(G)} m_\pi(\eta)\,  e^{\eta}$, with $m_\pi(\eta) \in \N_0$ the multiplicity of $\eta$ in $\pi$. 
The identification satisfies $\chi_\pi(\exp(X)) = \sum_{\eta \in P(G)} m_\pi(\eta)\, e^{\eta(X)}$ for all $X\in \mathfrak t_0$. 
Of course, $m_\pi(\eta)=0$ for all but finitely many $\eta\in P(G)$. 
For $\eta \in P(G)$, we set 
\begin{equation}\label{eq:xi}
	\xi_{G}(\eta) = \sum_{\omega\in W_{\mathfrak g}} \op{sgn}(\omega) \, e^{\omega(\eta)}. 
\end{equation}

For $\lambda\in P^{++}(G)$, let $\pi_\lambda$ denote the irreducible representation of $G$ with highest weight $\lambda$. 
The Weyl character formula ensures that 
\begin{equation}\label{eq:Weylcharacter}
	\chi_{\pi_\lambda} = \frac{\xi_G(\lambda+\rho_{\mathfrak g})} {\xi_G(\rho_{\mathfrak g})}. 
\end{equation}
Furthermore, 
\begin{equation}\label{eq:Weyldenominator}
	\xi_G(\rho_{\mathfrak g}) = \prod_{\alpha\in \Phi^+(\mathfrak g,\mathfrak t)} (e^{ \alpha/2} - e^{-\alpha/2}). 
\end{equation}

Let $K$ be a closed subgroup of $G$.
Suppose that a maximal torus $S$ of $K$ is contained in $T$. 
For $\beta\in \mathfrak t^*$, we denote by $\bar\beta$ its restriction to $\mathfrak s^*$. 
We extend this operator to the formal power series discussed above by setting $\overline{e^{\eta}}=e^{\bar \eta}$ for any $\eta\in P(G)$.
It turns out that, for a representation $\pi$ of $G$, the character of its restriction $\pi|_K$ to $K$ satisfies
\begin{equation}\label{eq:characterrestriction}
	\chi_{\pi|_K} = \overline{\chi_\pi}. 
\end{equation}

\subsection{Kostant's Branching Formula} \label{subsec:Kostant}
We will follow \cite[\S IX.4]{Knapp-book-beyond} (see also \cite[\S 8.2]{GoodmanWallach-book-Springer}). 
This formula is valid for a big amount of homogeneous spaces including all symmetric spaces.

Let $G$ be a connected compact Lie group, and let $K$ be a connected closed subgroup.
We assume that the centralizer $T$ in $G$ of a maximal torus $S$ of $K$ is abelian.
Thus, $T$ is a maximal torus in $G$.
Equivalently, there is a regular element of $K$ that is regular in $G$.
This allows us to introduce compatible positive systems $\Phi^+(\mathfrak g,\mathfrak t)$ and $\Phi^+(\mathfrak k,\mathfrak s)$ by defining positivity relative to a regular element in $\mi \mathfrak s_0$.
Set
\begin{equation}
\Sigma = \overline{ \Phi^+(\mathfrak g,\mathfrak t)} \smallsetminus \Phi^+(\mathfrak k,\mathfrak s). 
\end{equation}
More precisely, $\Sigma$ is the multiset given by the elements $\overline\alpha$ for $\alpha\in \Phi^+(\mathfrak g,\mathfrak t)$, repeated according to their multiplicity, but deleting the elements in $\Phi^+(\mathfrak k, \mathfrak s)$, each with multiplicity one. 
The Kostant partition function $\mathcal P_\Sigma$ is defined as follows: 
$\mathcal P_\Sigma(\nu)$ is the number of ways that a member $\nu$ of $\mathfrak s^*$ can be written as a sum of members of $\Sigma$, with the multiple versions of a member of $\Sigma$ being regarded as distinct.

Under the notation above, for $\lambda\in P^{++}(G)$ and $\mu\in P^{++}(K)$, Kostant's Branching Formula tells us that the multiplicity of the irreducible representation $\sigma_\mu$ of $K$ with highest weight $\mu$ in the restriction of the irreducible representation $\pi_\lambda$ of $G$ with highest weight $\lambda$ is given by
\begin{equation}\label{eq:Kostant}
\dim \Hom_{K}(\sigma_\mu ,\pi_\lambda)
= \sum_{w\in W_{\mathfrak g}} \op{sgn}(\omega) \;  \mathcal P_{\Sigma}(\overline{\omega(\lambda+\rho_{\mathfrak g})-\rho_{\mathfrak g}}-\mu).
\end{equation}

\subsection{Generalized Clebsch--Gordan formula} \label{subsec:Clebsh-Gordan} 
We now recall the extension of the Clebsch--Gordan formula given by Wallach and Yacobi~\cite{WallachYacobi09} for an arbitrary tensor product of irreducible representations of $\SU(2)$.
It is well known that the representations of $\SU(2)$ are parametrized by non-negative integer numbers.
We denote them by $\tau_{k/2}$ for $k\in \Z_{\geq0}$, whose underlying vector space $V_{\tau_{k/2}}$ has dimension $k+1$.

Let $\{\varepsilon_1,\dots,\varepsilon_{n+1}\}$ denote the canonical basis of $\C^{n+1}$. 
We set 
\begin{equation}\label{eq:Sigma'}
\Sigma'=\{\varepsilon_i\pm \varepsilon_{n+1}: 1\leq i\leq n\}. 
\end{equation}
For $\nu \in \C^{n+1}$, denote by $\mathcal P_{\Sigma'}(\nu )$ the number of ways of writing $\nu $ as a sum of elements in $\Sigma'$, 
\begin{equation}
\mathcal P_{\Sigma'}(\nu ) = \#\left\{ \{a_\alpha\}_{\alpha\in\Sigma'}: a_\alpha\in\N_0,\, \sum_{\alpha\in\Sigma'} a_\alpha \alpha=\nu  \right\}. 
\end{equation}

\begin{theorem}\cite[Thm.~2.3]{WallachYacobi09} \label{thm:ClebschGordan}
For non-negative integers $r_1,\dots,r_{n+1}$, the number of times that $\tau_{k/2}$ appears in $\tau_{r_1/2}\otimes \dots\otimes \tau_{r_{n+1}/2}$ is equal to
\begin{equation*}
\dim\Hom_{\SU(2)} ({ \tau_{k/2}}, \otimes_{i=1}^{n+1}  {\tau_{r_i/2}} )
= \textstyle \mathcal P_{\Sigma'}\left( \sum\limits_{i=1}^{n+1} r_i\varepsilon_i - k\varepsilon_{n+1}\right) - \mathcal P_{\Sigma'}\left( \sum\limits_{i=1}^{n+1} r_i\varepsilon_i + (k+2)\varepsilon_{n+1}\right).
\end{equation*}
\end{theorem}

Write $\intervalo =\{m\in\Z: 1\leq m\leq n\}$, and for any $I\subset \intervalo$, set $\beta_I= \sum_{i\in I}\varepsilon_i$.
The following elementary lemma will be very useful in the sequel.

\begin{lemma} \label{lem:P_Sigma'}
We have that 
$
\sum\limits_{I\subset \intervalo}\mathcal P_{\Sigma'} ( \nu -\beta_I) = \mathcal P_{\Sigma'}( 2\nu )
$
for every $\nu \in\C^{n+1}$. 
\end{lemma}	
\begin{proof}
Write $u^{\pm}_i=\varepsilon_i\pm \varepsilon_{n+1}$ for $i=1,\dots,n$. Let 
\begin{align}
\mathcal A&=\left\{ (a_1,b_1,\dots,a_n,b_n)\in \N_0^{2n}:\textstyle \sum\limits_{i=1}^n a_i u_i^+ + b_iu_i^-=\nu-\beta_I \text{ for some } I\subset \intervalo \right\};\\
\mathcal B&=\left\{ (c_1,d_1,\dots,c_n,d_n) \in \N_0^{2n}: \textstyle \sum\limits_{i=1}^n c_i u_i^+ + d_iu_i^-=2\nu \right\}.
\end{align}
We will prove that there exists a bijective correspondence between $\mathcal A$ and $\mathcal B$.

For $(a_1,b_1,\dots,a_n,b_n) \in\mathcal A$, we have $a_i+b_i=\nu_i-1$ for all $i\in I$, $a_i+b_i=\nu_i$ for all $i\notin I$, and $\sum_{i=1}^n a_i-b_i=\nu_{n+1}$, for some $I \subset\intervalo$.
Define
\begin{align}
c_i=&
\begin{cases} 
2a_i+1  &\quad\mbox{if } i\in I,\\
2a_i  &\quad\mbox{if } i\notin I;
\end{cases} & \text{and} & &
d_i=&
\begin{cases} 
2b_i+1 & \quad\mbox{if } i\in I,\\
2b_i  &\quad\mbox{if } i\notin I.
\end{cases}
\end{align}
Hence, $(c_1,d_1,\dots,c_n,d_n) \in\mathcal B$, since  
 $2c_i+2d_i=2\nu_i$ for all $i$ and $\sum_{i=1}^n c_i-d_i=2\nu_{n+1}$.

Moreover, if $(c_1,d_1,\dots,c_n,d_n)\in\mathcal B$, then for each $i=1,\dots,n$ we have $c_i\equiv d_i \pmod 2$. This implies that the correspondence is bijective and the lemma follows.
\end{proof}

\subsection{Representations of U(3)} \label{subsec:U(3)}
We now fix the notation to parametrize the irreducible representations of $\U(3)=\{g\in \GL(3,\C): g^*g=I_3\}$.
Let $T'=\{\diag(e^{\mi \theta_1}, e^{\mi \theta_2}, e^{\mi \theta_3}) : \theta_j\in\R\;\forall j\}$ with associated Cartan subalgebra $\mathfrak h'$ of $\mathfrak u(3)$ given by  $\mathfrak h'= \{\diag({\theta_1}, {\theta_2}, {\theta_3}) : \theta_j\in\C\;\forall j\}$. 
Let $\varepsilon_j'\in (\mathfrak h')^*$ for $1\leq j\leq 3$ given by $\varepsilon_j'(\diag({\theta_1}, {\theta_2}, {\theta_3}))=\theta_j$. 

We consider the standard order given by the lexicographic order with respect to the ordered basis $\{\varepsilon_1',\varepsilon_2', \varepsilon_3'\}$. 
Thus, the irreducible representations of $\U(3)$ are in correspondence with $P^{++}(\U(3)) = \{\sum_{j=1}^3 \lambda_j'\varepsilon_j': a_j\in\Z\;\forall \, j,\; a_1\geq a_2\geq a_3\}$. 
For $\lambda'\in P^{++}(\U(3))$, let $\pi_{\lambda'}'$ denote the irreducible representation of $\U(3)$ with highest weight $\lambda'$.

\section{Type B case} \label{sec:casoB}
Throughout this section, we set
\begin{align*}
G&=\SO(2n+3), &K&=\SO(2n),&  H&=\SO(3),
\end{align*}
for any $n\geq2$.
We have that $\mathfrak g=\so(2n+3,\C)$ is a classical Lie algebra of type B$_{n+1}$.

\subsection{Root system notation for type B case} \label{subsec2n:notation}
We first fix compatible notation for the corresponding root systems associated to $G$, $K$, $H$, and $K\times H$. 
We pick the maximal torus of $G$ given by
\begin{equation}\label{eq2n:maximaltorus}
T:=\{\diag ( R(\theta_1), \dots, R(\theta_{n+1}),1 ): \theta_j\in\R \; \forall\, j\}, 
\end{equation}
where $R(\theta)=\left(\begin{smallmatrix}\cos\theta & \sin\theta\\ -\sin\theta & \cos\theta \end{smallmatrix}\right)$, whose associated Cartan subalgebra is given by
\begin{equation}
\mathfrak t:= \left\{
\diag\left( 
\left(\begin{smallmatrix}0& i\theta_1 \\ -i\theta_1& 0 \end{smallmatrix} \right)
,\dots,
\left(\begin{smallmatrix}0&i\theta_{n+1}\\ i\theta_{n+1}&0 \end{smallmatrix} \right),0 \right) : \theta_j\in \C\;\forall j\right\}. 
\end{equation}
For any $1\leq j\leq n+1$, let $\varepsilon_j\in \mathfrak t^*$ given by $\varepsilon_j(X)=\theta_j$ for $X$ in $\mathfrak t$ as above. 
Then, the set of roots is 
$
\Phi(\mathfrak g,\mathfrak t)= \{\pm \varepsilon_i\pm\varepsilon_j: 1\leq i<j\leq n+1\}\cup \{\pm \varepsilon_j: 1\leq j\leq n+1\}. 
$

The maximal torus $T\cap K$ of $K$  satisfies $(\mathfrak k\cap \mathfrak t)^* = \op{span}_\C\{\varepsilon_1,\dots,\varepsilon_{n}\}$, thus $\Phi(\mathfrak k,\mathfrak t)= \{\pm \varepsilon_i\pm\varepsilon_j: 1\leq i<j\leq n\}$. 
Similarly, $T\cap H$ is a maximal torus in $H$ satisfying that $(\mathfrak h\cap \mathfrak t)^*=\op{span}_\C \{\varepsilon_{n+1}\}$, thus $\Phi(\mathfrak h,\mathfrak t)= \{\pm \varepsilon_{n+1}\}$.

We fix the order on $\mathfrak t^*$ given by the lexicographic order with respect to the ordered basis $\{\varepsilon_1, \dots,\varepsilon_{n+1}\}$. 
We have compatible order on $(\mathfrak k\cap \mathfrak t)^*$ and $(\mathfrak h\cap \mathfrak t)^*$. 
Then 
\begin{align}
\label{eq2n:positiverootsG}
\Phi^+(\mathfrak g,\mathfrak t) &= \{ \varepsilon_i\pm\varepsilon_j: 1\leq i<j\leq n+1\}\cup \{\varepsilon_j: 1\leq j\leq n+1\},\\
\label{eq2n:positiverootsKH}
\Phi^+(\mathfrak k,\mathfrak t) &= \{ \varepsilon_i\pm\varepsilon_j: 1\leq i<j\leq n\}, \qquad
\Phi^+(\mathfrak h,\mathfrak t) = \{\varepsilon_{n+1}\},
\end{align}
\begin{align}
P(G) &= \oplus_{j=1}^{n+1} \Z\varepsilon_j,&
P^{++}(G) &= \{\textstyle\sum_{j=1}^{n+1} \lambda_j\varepsilon_j\in P(G): \lambda_1\geq \dots\geq \lambda_{n+1}\geq 0\}, \\
P(K) &= \oplus_{j=1}^{n} \Z\varepsilon_j,&
P^{++}(K) &= \{\textstyle\sum_{j=1}^{n} \lambda_j\varepsilon_j\in P(K): \lambda_1\geq \dots\geq \lambda_{n-1}\geq |\lambda_{n}|\},\\
P(H) &=  \Z\varepsilon_{n+1},&
P^{++}(H) &= \{k\varepsilon_{n+1}\in P(H): k\geq0\},
\end{align}
\begin{align}
	\rho_{\mathfrak g} &:= \sum_{i=1}^{n+1} (n+\tfrac32-i)\varepsilon_i, &
	\rho_{\mathfrak k} &:= \sum_{i=1}^{n} (n-i)\varepsilon_i, &
	\rho_{\mathfrak h} &:= \tfrac12\varepsilon_{n+2}.
\end{align}
We will denote by $\pi_\lambda, \sigma_\mu,\tau_{k\varepsilon_{n+1}}$ the irreducible representations of $G$, $K$ and $H$ with highest weights $\lambda\in P^{++}(G)$, $\mu\in P^{++}(K)$, and $k\varepsilon_{n+1}\in P^{++}(H)$, respectively. 
We will abbreviate $\tau_k=\tau_{k\varepsilon_{n+1}}$.

We now describe the Weyl group $W_{\mathfrak g}$. 
Any element $\omega\in W_{\mathfrak g}$ is of the form $\omega=s p$, with $p$ a permutation of the $n+1$ coordinates and $s$ a multiplication by $-1$ on a subset of coordinates. 
For $1\leq i\leq n+1$, write $s_i:\mathfrak t^*\to \mathfrak t^*$ the reflexion with respect to the axis $i$, that is, $s_i(\varepsilon_i)= -\varepsilon_i$ and $s_i(\varepsilon_j)= \varepsilon_j$ for all $j\neq i$.

We consider the inner product $\langle\cdot,\cdot\rangle$ on $\mathfrak g$ given by $\langle X,Y\rangle =\tfrac12 \tr(XY)$. 
With respect to its extension to $\mathfrak t^*$,  $\{\varepsilon_1,\dots,\varepsilon_{n+1}\}$ is an orthonormal basis.

\subsection{Main theorem for type B case}
The main result in this section is the following.

\begin{theorem}\label{thm2n:mainthm}
Let $n\geq2$, $G=\SO(2n+3)$, $K=\SO(2n)$, $H=\SO(3)$, $\lambda=\sum_{i=1}^{n+1}\lambda_i \varepsilon_i\in P^{++}(G)$, $\mu=\sum_{i=1}^n\mu_i \varepsilon_i\in P^{++}(K)$.
If $\mu$ simply interlaces $\lambda$, i.e.\ $\lambda_i\geq |\mu_i|\geq \lambda_{i+1}$ for $1\leq i\leq n$, then
\begin{equation}\label{eq2n:tensores}
\Hom_K({\sigma_\mu},{\pi_\lambda}) \simeq \tau_{\lambda_{n+1}} \otimes\,  \bigotimes_{j=1}^n \left( \bigoplus_{m=0}^{\lfloor (\lambda_j-|\mu_j|)/2\rfloor }  \tau_{\lambda_j-|\mu_j|-2m} \right)
\end{equation}
as $H$-modules. 
\end{theorem}

The absolute value on every $\mu_i$ simplifies the notation, though $|\mu_i|=\mu_i$ for all $i<n$. 

Kostant's Branching Formula \ref{eq:Kostant} will be the main tool to prove Theorem~\ref{thm2n:mainthm}. 
We next give the first steps to apply it to the symmetric space $G/(K\times H)$.
The maximal torus $T$ in $G$ (see \eqref{eq2n:maximaltorus}) is also a maximal torus in $K\times H$. 
In particular, the restriction denoted by a bar is the identity operator on $\mathfrak t^*$. 
Furthermore, $\Phi^+(\mathfrak k\times\mathfrak h,\mathfrak t)=\Phi^+(\mathfrak k,\mathfrak t\cap\mathfrak k)\cup \Phi^+(\mathfrak h,\mathfrak t \cap\mathfrak h)$.
From \eqref{eq2n:positiverootsG} and \eqref{eq2n:positiverootsKH}, it follows that
\begin{equation}
\Sigma=  \{\varepsilon_i\pm \varepsilon_{n+1}: 1\leq i\leq n\} \cup \{\varepsilon_i: 1\leq i\leq n\},
\end{equation}
each element with multiplicity one. 
Clearly, for $\nu \in\mathfrak t^*$, 
\begin{equation}\label{eq2n:P_Sigma}
\mathcal P_\Sigma(\nu)>0 
\quad \Longrightarrow \quad
\nu_i:= \langle\varepsilon_i,\nu \rangle \in \Z \quad\forall i,\; \nu_i\geq0 \quad\forall \,1\leq i\leq n, \; |\nu_{n+1}|\leq \textstyle \sum\limits_{i=1}^n\nu_i. 
\end{equation}

Let $\lambda=\sum_{i=1}^{n+1}\lambda_i \varepsilon_i\in P^{++}(G)$, $\mu=\sum_{i=1}^n\mu_i \varepsilon_i\in P^{++}(K)$, and $k\varepsilon_{n+1}\in P^{++}(H)$. 
We obtain from \eqref{eq:Kostant} that the number of times that $\sigma_\mu\otimes \tau_k\in \widehat {K\times H}$ occurs in $\pi_\lambda|_{K\times H}$ is given by 
\begin{equation}\label{eq2n:Kostant}
\dim \Hom_{K\times H}({\sigma_\mu}\otimes {\tau_k},{\pi_\lambda})
= \sum_{\omega \in W_{\mathfrak g}} \op{sgn}(\omega) \;  \mathcal P_{\Sigma} \big( \, \overline{ \omega(\lambda+\rho_{\mathfrak g})- \rho_{\mathfrak g}} -\mu- k\varepsilon_{n+1}\big).
\end{equation}

The next lemmas will indicate the set of elements $\omega\in W_{\mathfrak g}$ such that the $\omega$-th term in the above formula does not vanish.

\begin{lemma}\label{lem2n:Weyl1}
Assume $\mu_n\geq0$. 
If $\omega \in W_{\mathfrak g}$ satisfies that the $\omega$-th term in \eqref{eq2n:Kostant} is non-zero, then $\omega=p $ or $\omega= s_{n+1}p$, for some permutation $p$. 
\end{lemma}

\begin{proof}
Write $\omega =sp$ as above. 
Suppose that $s(\varepsilon_j)=-\varepsilon_j$ for some $1\leq j\leq n$.
Since $\mathcal P_{\Sigma}(\omega(\lambda+\rho_{\mathfrak g})-\rho_{\mathfrak g}-\mu-k\varepsilon_{n+1})>0$ by assumption, \eqref{eq2n:P_Sigma} forces $0\leq \langle \varepsilon_j, \omega(\lambda+\rho_{\mathfrak g})-\rho_{\mathfrak g}-\mu-k\varepsilon_{n+1}\rangle$, so $0<\langle \varepsilon_j, \rho_{\mathfrak g}+\mu+k\varepsilon_{n+1}\rangle \leq \langle \varepsilon_j, \omega(\lambda+\rho_{\mathfrak g}) \rangle = \langle \varepsilon_j, sp(\lambda+\rho_{\mathfrak g}) \rangle = -\langle \varepsilon_j, p(\lambda+\rho_{\mathfrak g}) \rangle \leq 0$, which is a contradiction. 
Hence, $s(\varepsilon_j)=\varepsilon_j$ for all $1\leq j\leq n$ and the claim follows. 
\end{proof}

\begin{lemma}\label{lem2n:Weyl2}
Assume $\lambda_i\geq \mu_i\geq \lambda_{i+1}$ for all $1\leq i\leq n$. 
If $w\in W_{\mathfrak g}$ satisfies that the $\omega$-th term in \eqref{eq2n:Kostant} is non-zero, then $\omega=1$ or $\omega=s_{n+1}$.
\end{lemma}
\begin{proof}
By Lemma~\ref{lem2n:Weyl1}, we may assume that $w=sp$ with $p$ a permutation and $s=s_{n+1}$ or $s=1$. 
Let $l$ be an index distinct to $n+1$ and let $r$ be the index such that $p(\varepsilon_r)=\varepsilon_l$. 
Since $\mathcal P_{\Sigma}(\omega(\lambda+\rho_{\mathfrak g})-\rho_{\mathfrak g}-\mu-k\varepsilon_{n+1})>0$ by assumption, \eqref{eq2n:P_Sigma} gives $0\leq \langle \varepsilon_l, \omega(\lambda+\rho_{\mathfrak g})-\rho_{\mathfrak g}-\mu-k\varepsilon_{n+1}\rangle = \langle \varepsilon_r, \lambda+\rho_{\mathfrak g}\rangle -\langle \varepsilon_l,\rho_{\mathfrak g}+\mu+k\varepsilon_{n+1}\rangle = \lambda_r-\mu_l+(l-r)$. 
We conclude that $\lambda_r\geq \mu_l +r-l$. 
		
If $l=1$, then $r=1$. 
Indeed, if $r\geq2$, then $\lambda_2\geq \lambda_r\geq \mu_1+r-1>\lambda_2$ by assumption, which is a contradiction. 
Suppose $l\leq n-1$ and $\omega(\varepsilon_i)=\varepsilon_i$ for all $1\leq i<l$, and let $r$ be an index such that $\omega(\varepsilon_r)=\varepsilon_l$. 
Clearly, $r\geq l$. 
If $r\geq l+1$, then $\lambda_{l+1}\geq \lambda_r\geq \mu_l+r-l>\mu_l\geq \lambda_{l+1}$ by assumption, which is again a contradiction. 
We conclude that $r=l$. 
We have shown that $\omega(\varepsilon_i)=\varepsilon_i$ for all $1\leq i\leq n-1$.
In other words, $\omega=s$ or $\omega=sp_{n,n+1}$ where $p_{n,n+1}$ preserves $\varepsilon_i$ for all $1\leq i\leq n-1$ and switches $\varepsilon_n$ and $\varepsilon_{n+1}$. 
If $\omega=sp_{n,n+1}$, then $0\leq \langle \varepsilon_n, \omega(\lambda+\rho_{\mathfrak g})-\rho_{\mathfrak g}-\mu-k\varepsilon_{n+1}\rangle= \lambda_{n+1}-1-\mu_n$ which is impossible when $\mu_n\geq0$ because $|\mu_n|\geq \lambda_{n+1}$.
This completes the proof. 
\end{proof}

We now show that it is sufficient to prove Theorem~\ref{thm2n:mainthm} for $\mu_n\geq0$. 
We set $\widetilde\mu=\sum_{i=1}^{n-1} \mu_i\varepsilon_i- \mu_n\varepsilon_n$ for any $\mu=\sum_{i=1}^{n} \mu_i\varepsilon_i \in P^{++}(K)$. 
Note $\widetilde \mu \in P^{++}(K)$.

\begin{lemma}\label{lem2n:mu_n>0}
For any $\lambda\in P^{++}(G)$, $\mu\in P^{++}(K)$, $k\geq0$, $\sigma_{\mu}\otimes \tau_k$ and $\sigma_{\widetilde \mu}\otimes \tau_k$ occur in $\pi_{\lambda}|_{K\times H}$ the same number of times. 
\end{lemma}

\begin{proof}
We set $g_0\in \diag(1,\dots,1,-1,1,1,1) \in \Ot(2n+3)$. 
Although $g_0$ is not in $G$, the map $\varphi: x \mapsto g_0xg_0$ is an automorphism of $G$. 
It turns out that $\pi_\lambda\circ\varphi \simeq\pi_\lambda$, $\sigma_\mu\circ\varphi|_K \simeq \sigma_{\widetilde \mu}$, and $\tau_k\circ\varphi|_H = \tau_k$. 
It follows immediately that $\Hom_{K\times H}({\sigma_\mu}\otimes {\tau_k}, {\pi_\lambda}) \simeq \Hom_{K\times H}({\sigma_{\widetilde \mu}}\otimes {\tau_k}, {\pi_\lambda})$ as complex vector spaces, as asserted. 
\end{proof}

We are now in position to prove the main theorem of this section.

\begin{proof}[Proof of Theorem~\ref{thm2n:mainthm}]
To establish the isomorphism of $H$-modules in \eqref{eq2n:tensores}, we will show that each irreducible representation $\tau_k$ of $H$ occurs in both sides with the same multiplicity. 
For the right-hand side, we will use Theorem~\ref{thm:ClebschGordan}. 
For the left-hand side, we use that 
\begin{equation}\label{eq2n:branchingKxH}
\dim \Hom_H({\tau_k}, \Hom_K({\sigma_\mu},{\pi_\lambda})) = \dim \Hom_{K\times H}({\sigma_\mu}\otimes {\tau_k} ,{\pi_\lambda}),
\end{equation}
and then we apply Kostant's Branching Formula  to $G/(K\times H)$. 
Note that \eqref{eq2n:branchingKxH} and Lemma~\ref{lem2n:mu_n>0} allow us to assume $\mu_n\geq0$ for the rest of the proof.

Lemmas~\ref{lem2n:Weyl1} and \ref{lem2n:Weyl2} tell us that the only non-zero terms in \eqref{eq2n:Kostant} are given by $\omega=1$ and $\omega= s_{n+1}$. 
The term corresponding to the last choice is equal to
\begin{multline}
\mathcal P_{\Sigma} \big( \overline{ s_{n+1} (\lambda+\rho_{\mathfrak g})- \rho_{\mathfrak g}} -\mu- k\varepsilon_{n+1}\big)
=\mathcal P_{\Sigma} \Big({\textstyle \sum\limits_{i=1}^n\lambda_i\varepsilon_i}  -\mu- (\lambda_{n+1}+k+1)\varepsilon_{n+1}\Big)
\\
=\mathcal P_{\Sigma} \Big(s_{n+1} \big({\textstyle\sum\limits_{i=1}^n\lambda_i\varepsilon_i}  -\mu+ (\lambda_{n+1}+k+1)\varepsilon_{n+1} \big)\Big)
=\mathcal P_{\Sigma} \big(\lambda-\mu+ (k+1)\varepsilon_{n+1}\big).
\end{multline}
The last identity follows by $\mathcal P_{\Sigma}(s_{n+1}v)=\mathcal P_{\Sigma}(v)$ for all $v$. 
We have established so far that 
\begin{multline} 
\dim \Hom_{H}({\tau_k} ,\Hom_{K}({\sigma_\mu}, {\pi_\lambda})) 
= \mathcal P_{\Sigma}(\lambda-\mu-k\varepsilon_{n+1}) - \mathcal P_{\Sigma} (\lambda-\mu+ (k+1)\varepsilon_{n+1}). 
\end{multline}
Since $\Sigma'=\{\varepsilon_i\pm \varepsilon_{n+1}: 1\leq i\leq n\}\subset\Sigma$, it follows that the previous expression becomes
\begin{multline}
\sum_{\nu_1=0}^{\lambda_1-\mu_1} \dots \sum_{\nu_n=0}^{\lambda_n-\mu_n} \Big( \mathcal  P_{\Sigma'} \big(\lambda-\mu-k\varepsilon_{n+1}- \nu \big) 
- \mathcal P_{\Sigma'}  \big(\lambda-\mu+ (k+1)\varepsilon_{n+1} - \nu \big) \Big) \\
= \sum_{\gamma_1=0}^{\lfloor \frac{\lambda_1-\mu_1}{2}\rfloor} \dots \sum_{\gamma_n=0}^{\lfloor \frac{\lambda_n-\mu_n}{2}\rfloor} \left( \sum_{I\subset\intervalo} \mathcal P_{\Sigma'} \big(\lambda-\mu-k\varepsilon_{n+1} -2\gamma - \beta_I \big) 
\right. \\ \left.
- \sum_{I\subset\intervalo} \mathcal P_{\Sigma'} \big(\lambda-\mu + (k+1)\varepsilon_{n+1}-2\gamma-\beta_I \big) \right).
\end{multline}
Here and subsequently, we write $\nu={\textstyle\sum_{i=1}^n \nu_i\varepsilon_i}$ and $\gamma={\textstyle \sum_{i=1}^n \gamma_i\varepsilon_i}$, whenever $\nu_i$ and $\gamma_i$ for $1\leq i\leq n$ are determined. 
The notation for $\beta_I$ was introduced before Lemma~\ref{lem:P_Sigma'}. 
Such lemma implies 
\begin{multline}\label{eq2n:tauLHS}
\dim \Hom_{H}({\tau_k} ,\Hom_{K}({\sigma_\mu}, {\pi_\lambda}))
\\ =
\sum_{\gamma_1=0}^{\lfloor \frac{\lambda_1-\mu_1}{2}\rfloor} \dots \sum_{\gamma_n=0}^{\lfloor \frac{\lambda_n-\mu_n}{2}\rfloor} \Big(\mathcal P_{\Sigma'} \big(2(\lambda-\mu-k\varepsilon_{n+1} -2\gamma)  \big) 
- \mathcal P_{\Sigma'} \big(2(\lambda-\mu+ (k+1)\varepsilon_{n+1}-2\gamma) \big) \Big).
\end{multline}

On the other hand, the right-hand side of \eqref{eq2n:tensores} is 
\begin{equation}\label{eq2n:RHS}
\bigoplus_{\gamma_1=0}^{\lfloor \frac{\lambda_1-\mu_1}{2}\rfloor} 
\dots \bigoplus_{\gamma_n=0}^{\lfloor \frac{\lambda_n-\mu_n}{2}\rfloor}
\tau_{\lambda_1-\mu_1-2\gamma_1}\otimes\dots \otimes \tau_{\lambda_n-\mu_n-2\gamma_n} \otimes \tau_{\lambda_{n+1}}.
\end{equation}
One has that $H=\SO(3)\simeq \SU(2)/\{\pm 1\}$. 
The irreducible representation $\tau_{k/2}$ of $\SU(2)$ with $k$ odd does not descend to a representation of $\SO(3)$.
Furthermore, the $(2k+1)$-dimensional representation $\tau_{k}$ of $\SU(2)$ descends to $\SO(3)$.
Therefore, Theorem~\ref{thm:ClebschGordan} ensures that the number of times that $\tau_k$ occurs in the factor $\tau_{\lambda_1-\mu_1-2\gamma_1}\otimes\dots \otimes \tau_{\lambda_n-\mu_n-2\gamma_n} \otimes \tau_{\lambda_{n+1}}$ is equal to 
\begin{equation}
\textstyle \mathcal P_{\Sigma'}\left( 2\big(\lambda-\mu-2\gamma \big) -2k\varepsilon_{n+1}\right) 
- \mathcal P_{\Sigma'}\left( 2\big(\lambda-\mu -2\gamma\big)+ (2k+2)\varepsilon_{n+1}\right).
\end{equation}
It follows that the number of times that $\tau_k$ occurs in \eqref{eq2n:RHS} coincides with \eqref{eq2n:tauLHS}, as asserted. 
\end{proof}

\subsection{Prescribed highest weight end for type B case}
The aim in this subsection is to prove Theorem~\ref{thm2n:ending}. 
Roughly speaking, it is a decomposition of the multiplicity space $\Hom_K({\sigma_{\mu}},{\pi_\lambda})$ as $H$-module when the ending coefficients of $\lambda$ coincide with a part of the coefficients of $\mu$.

\begin{theorem}\label{thm2n:ending}
Let $G=\SO(2n+3)$, $K=\SO(2n)$, and $H=\SO(3)$ for any $n\geq2$.
For $\mu =\sum_{i=1}^{n}\mu_i \varepsilon_i\in P^{++}(K)$ and $\lambda=\sum_{i=1}^{n+1} \lambda_i\varepsilon_i \in P^{++}(G)$ with $\lambda_{i+3}=\mu_i$ for all $1\leq i\leq n-2$ and $\mu_{n-1}\leq \lambda_{n+1}$, we have that
\begin{equation}
\Hom_K(\sigma_\mu,\pi_\lambda)\simeq 
\pi'_{\lambda'}|_H\otimes \left(\bigoplus_{j=|\mu_n|}^{\mu_{n-1}} \tau_j\right)
\end{equation}
as $H$-modules, where $\pi'_{\lambda'}$ denotes the irreducible representation of $\U(3)$ with highest weight $\lambda': = \lambda_1\varepsilon_1'+\lambda_2\varepsilon_2'+ \lambda_3\varepsilon_3'$ (see Subsection~\ref{subsec:U(3)}). 
\end{theorem}

We note that the condition $\mu_{n-1}\leq \lambda_{n+1}$ is redundant unless $n=2$.

It turns out that Kostant's branching formula is not adequate to prove this result, since the number of non-zero terms in the formula \eqref{eq2n:Kostant} is too high. 
We will use the branching law from $G$ to $K\times H$ given by Tsukamoto~\cite{Tsukamoto05}.
We first recall such a result.

Fix $\lambda\in P^{++}(G)$. 
It will be convenient to set $\lambda_{n+2}=\lambda_{n+3}=0$ and $a_0=\mu_0=\lambda_1$.
Furthermore, for $\mu\in P^{++}(K)$, let $\widetilde \sigma_\mu$ denote the representation of $K$ given by $\widetilde \sigma_\mu = \sigma_\mu$ if $\mu_n=0$ and $\widetilde \sigma_\mu = \sigma_\mu\oplus \sigma_{\widetilde \mu}$ otherwise.

Tsukamoto established in the proof of Theorem~1 in \cite{Tsukamoto05} that
\begin{equation}\label{eq2n:Tsukamotocharacters}
\chi_{\pi_\lambda|_{K\times H}} = 
\frac{1}{(e^{\frac12{\varepsilon_{n+1}}}-  e^{-\frac12{\varepsilon_{n+1}}})} 
\sum_{\mu} \;\chi_{\widetilde \sigma_\mu}\;
\sum_{(a_1,\dots,a_n)} \frac{\prod\limits_{i=1}^{n+1} (e^{l_i\varepsilon_{n+1}}-e^{-l_i\varepsilon_{n+1}})} {(e^{\varepsilon_{n+1}}-  e^{-\varepsilon_{n+1}})^n  } ,
\end{equation}
where the first sum is over every $\mu\in P^{++}(K)$ \emph{triply interlacing} $\lambda$, that is, $\lambda_i\geq \mu_i\geq \lambda_{i+3}$ for all $1\leq i\leq n$, the second sum is over the $n$-tuples $(a_1,\dots,a_n)\in\Z^n$ satisfying that $a_1 \geq \dots \geq  a_n\geq0$ and $\max(\mu_i,\lambda_{i+2}) \leq a_i\leq \min(\mu_{i-1},\lambda_i)$ for all $1\leq i\leq n$, and the parameters $l_1,\dots,l_{n+1}$ are given by 
\begin{equation}\label{eq2n:l_i}
\begin{cases}
l_i = \min(\lambda_{i},a_{i-1}) - \max(\lambda_{i+1}, a_{i})+1 & \quad\text{for all $1\leq i\leq n$,}\\
l_{n+1} = \min(\lambda_{n+1}, a_n)+1/2.
\end{cases}
\end{equation}

Since $\chi_{\tau_k} = \xi_H(k\varepsilon_{n+1}+\rho_{\mathfrak h})/\xi_H(\rho_{\mathfrak h})$ by \eqref{eq:Weylcharacter}, with $\xi_H(k\varepsilon_{n+1}+\rho_{\mathfrak h})= \xi_H((k+\frac12)\varepsilon_{n+1})= e^{(k+\frac12)\varepsilon_{n+1}}-e^{-(k+\frac12)\varepsilon_{n+1}}$ and $\xi_H(\rho_{\mathfrak h})= e^{\frac12{\varepsilon_{n+1}}}-  e^{-\frac12{\varepsilon_{n+1}}}$,   
Tsukamoto thus obtained the following implicit branching rule from $G=\SO(2n+3)$ to $K\times H=\SO(2n) \times \SO(3)$.

\begin{theorem}\cite[Theorem~1]{Tsukamoto05}\label{thm2n:Tsukamoto}
Let $\lambda\in P^{++}(G)$, $\mu\in P^{++}(K)$, and $k\varepsilon_{n+1}\in P^{++}(H)$. 
If $\lambda_i\geq \mu_i\geq \lambda_{i+3}$ for all $1\leq i\leq n-1$ and $\lambda_n\geq |\mu_n|$, then the number of times that $\sigma_\mu\otimes \tau_k$ occurs in $\pi_\lambda|_{K\times H}$ is given by $m_{k}$, where the coefficients $m_p$ for $p\geq0$ are defined by
\begin{equation}\label{eq2n:Tsukamoto}
\sum_{(a_1,\dots,a_n)} \frac{\prod\limits_{i=1}^{n+1} (e^{l_i\varepsilon_{n+1}}-e^{-l_i\varepsilon_{n+1}})} {(e^{\varepsilon_{n+1}}-  e^{-\varepsilon_{n+1}})^n}  = \sum_{p\geq0} m_p \; \big(e^{(p+\frac12)\varepsilon_{n+1}}-e^{-(p+\frac12)\varepsilon_{n+1}}\big),
\end{equation}
where the sum at the left is over the $n$-tuples $(a_1,\dots,a_n)\in\Z^n$ satisfying that $a_1 \geq \dots \geq  a_n\geq0$, $\max(\mu_i,\lambda_{i+2}) \leq a_i\leq \min(\mu_{i-1},\lambda_i)$ for all $1\leq i\leq n-1$, $\max(|\mu_n|,\lambda_{n+2}) \leq a_n\leq \min(\mu_{n-1},\lambda_n)$, and $l_1,\dots,l_{n+1}$ are given by \eqref{eq2n:l_i}. 
Otherwise, $\sigma_\mu\otimes \tau_k$ does not occur in $\pi_\lambda|_{K\times H}$. 
\end{theorem}

\begin{proof}[Proof of Theorem~\ref{thm2n:ending}]
We will assume throughout the proof that $\mu_n\geq0$, which is possible by Lemma~\ref{lem2n:mu_n>0}.
We will first show that 
\begin{equation}\label{eq2n:reduccion}
\Hom_K(\sigma_\mu,\pi_\lambda)\simeq \Hom_K(\sigma_0,\pi_{\lambda'})\otimes \left(\bigoplus_{j=\mu_n}^{\mu_{n-1}} \tau_j\right)
\end{equation}
as $H$-modules, where $\lambda'=\lambda_1\varepsilon_1+ \lambda_2\varepsilon_2 + \lambda_3\varepsilon_3$. 
To do that, we will check that the term accompanying $\chi_{\sigma_\mu}$ in $\chi_{\pi_\lambda}$ coincides with the term accompanying $\chi_{\sigma_0}$ in $\chi_{\pi_{\lambda'}}$ times $\sum_{j=\mu_n}^{\mu_{n-1}} \chi_{\tau_j}$.

From \eqref{eq2n:Tsukamotocharacters}, the term accompanying $\chi_{\sigma_\mu}$ in $\chi_{\pi_\lambda}$ equals 
\begin{equation}
\frac{1}{\xi_H(\rho_{\mathfrak h})}
\sum_{(a_1,\dots,a_n)} \frac{\prod\limits_{i=1}^{n+1} (e^{l_i\varepsilon_{n+1}}-e^{-l_i\varepsilon_{n+1}})} {(e^{\varepsilon_{n+1}}-  e^{-\varepsilon_{n+1}})^n  },
\end{equation}
where the sum is over $(a_1,\dots,a_n)\in\Z^n$ satisfying that $a_1 \geq \dots \geq  a_n\geq0$ and $\max(\mu_i,\lambda_{i+2}) \leq a_i\leq \min(\mu_{i-1},\lambda_i)$ for all $1\leq i\leq n$, and the parameters $l_1,\dots,l_{n+1}$ are given by \eqref{eq2n:l_i}. 
By assumption, $\lambda_{i+3} = \mu_i$ for all $1\leq i\leq n-2$. 
This implies that the sum reduces to $(a_1,\dots,a_n)\in\Z^n$ satisfying that $\lambda_3\leq a_1\leq \lambda_1$, $\mu_n \leq a_n\leq \mu_{n-1}$, and $a_i=\lambda_{i+2}$ for all $1\leq i\leq n-1$, thus $l_1=\lambda_1+1-\max(\lambda_2,a_1)$, $l_2=\min(\lambda_2,a_1)-\lambda_3+1$, $l_i=1$ for all $3\leq i\leq n$, $l_{n+1}=a_n+1/2$. 
Hence, the previous expression becomes
\begin{multline}
\sum_{a_1= \lambda_3}^{\lambda_1}   \frac{(e^{l_1\varepsilon_{n+1}}-e^{-l_1\varepsilon_{n+1}}) (e^{l_2\varepsilon_{n+1}}-e^{-l_2\varepsilon_{n+1}}) } {(e^{\varepsilon_{n+1}}-  e^{-\varepsilon_{n+1}})^2}  \; \sum_{a_n=\mu_n}^{\mu_{n-1}}\, \frac{(e^{l_{n+1}\varepsilon_{n+1}}-e^{-l_{n+1}\varepsilon_{n+1}})} {\xi_H(\rho_{\mathfrak h})}  \\
= \left(\sum_{a_1= \lambda_3}^{\lambda_1} \,  \frac{(e^{l_1\varepsilon_{n+1}}-e^{-l_1\varepsilon_{n+1}}) (e^{l_2\varepsilon_{n+1}}-e^{-l_2\varepsilon_{n+1}}) } {(e^{\varepsilon_{n+1}}-  e^{-\varepsilon_{n+1}})^2} \right)\left(\sum_{a_n=\mu_n}^{\mu_{n-1}} \,\chi_{\tau_{a_n}} \right),
\end{multline}
where $l_1=\lambda_1+1-\max(\lambda_2,a_1)$ and  $l_2=\min(\lambda_2,a_1)-\lambda_3+1$.
We conclude that the proof of \eqref{eq2n:reduccion} follows by checking that the second term in the right-hand side of the last expression  coincides with the term accompanying $\chi_{\sigma_0}$ in $\chi_{\pi_{\lambda'}}$. 
This can be easily checked by \eqref{eq2n:Tsukamotocharacters} since $\lambda_{i+3}'=\mu_i=0$ for all $i\geq 1$. 
Indeed, one obtains $\lambda_3\leq a_1\leq \lambda_1$, $a_i=0$ for all $2\leq i\leq n$, $l_1=\lambda_1+1-\max(\lambda_2,a_1)$,  $l_2=\min(\lambda_2,a_1)-\lambda_3+1$, $l_i=1$ for all $3\leq i\leq n$, and $l_{n+1} = 1/2$. 

He have established so far the identity \eqref{eq2n:reduccion}. 
Now, by using Knapp's duality \cite{Knapp01}, the multiplicity space $\Hom_K(\sigma_0,\pi_{\lambda'})$ is isomorphic as an $H$-module to the restriction of the representation $\pi_{\lambda'}'$ of $\U(3)$ to $H=\SO(3)$. 
This completes the proof. 
\end{proof}

\begin{remark}
The branching law from $\U(3)$ to $\SO(3)$ has been thoroughly studied.
It states that (see for instance \cite[(1.18)]{GalbraithLouck91}) the number of times that $\tau_k$ occurs in $\pi_{\lambda'}'|_{\SO(3)}$ with $\lambda'=\sum_{j=1}^3 \lambda_j'\varepsilon_j'$, is given by 
\begin{equation}\label{eq:U(3)toSO(3)}
	\begin{cases}
		0 
		& \quad \text{if } 0\leq p \leq k-1,\\
		\lceil\tfrac{p-k+1}{2}\rceil - \lceil\tfrac{p-k-q}{2}\rceil
		& \quad \text{if } k\leq p\leq 2k,\quad 0\leq q \leq p-k,\\
		\lceil\tfrac{p-k+1}{2}\rceil 
		& \quad \text{if } k\leq p\leq 2k,\quad p-k\leq q \leq k,\\
		\lceil\tfrac{p-k+1}{2}\rceil - \lceil\tfrac{q-k}{2}\rceil
		& \quad \text{if } k\leq p\leq 2k,\quad k\leq q ,\\
		\lceil\tfrac{p-k+1}{2}\rceil - \lceil\tfrac{p-k-q}{2}\rceil
		& \quad \text{if } 2k\leq p,\quad 0\leq q \leq k,\\
		\lceil\tfrac{p-k+1}{2}\rceil - \lceil\tfrac{p-k-q}{2}\rceil - \lceil\tfrac{q-k}{2}\rceil
		& \quad \text{if } 2k\leq p,\quad k\leq q \leq p-k,\\
		\lceil\tfrac{p-k+1}{2}\rceil - \lceil\tfrac{q-k}{2}\rceil
		& \quad \text{if } 2k\leq p,\quad p-k\leq q ,
	\end{cases}
\end{equation} 
where $p=\lambda_1'-\lambda_3'$,  $q=\lambda_2'-\lambda_3'$, and $\lceil\cdot \rceil$ stands for the ceiling function (i.e.\ $\lceil x \rceil$ is the least integer greater than or equal to $x$).
Now, Theorem~\ref{thm2n:ending} and \eqref{eq:U(3)toSO(3)} give an explicit expression for the number of times that $\sigma_\mu\otimes \tau_k$ occurs in $\pi_{\lambda}|_{K\times H}$ for $\lambda$ and $\mu$ as in Theorem~\ref{thm2n:ending}. 
\end{remark}

\begin{remark}
Knapp's duality \cite{Knapp01} was already present in \cite{GrossKunze84} (cf.\ second-to-last paragraph in \S2 of \cite{HoweTanWillenbring05}). 
\end{remark}

\section{Type D case} \label{sec:casoD}
We consider in this section the case $d=2n+1$, thus for any $n\geq1$, we set 
\begin{align*}
	G&=\SO(2n+4), &K&=\SO(2n+1),&  H&=\SO(3).
\end{align*}
The Lie algebra $\mathfrak g=\so(2n+4,\C)$ is a classical Lie algebra of type D$_{n+2}$.
There are several similarities with the previous case considered in Section~\ref{sec:casoB}, so we will omit many details.

\subsection{Root system notation for type D case} \label{subsec2n+1:notation}
We pick the maximal torus
\begin{align}\label{eq2n+1:toromaximal}
	T&:=\{\diag ( R(\theta_1), \dots, R(\theta_{n+2}) ): \theta_j\in\R \; \forall\, j\}, 
\\
\mathfrak t &:= \left\{
\diag\left( 
\left(\begin{smallmatrix}0& i\theta_1 \\ -i\theta_1& 0 \end{smallmatrix} \right)
,\dots,
\left(\begin{smallmatrix}0&i\theta_{n+2}\\ i\theta_{n+2}&0 \end{smallmatrix} \right) \right) : \theta_j\in \C\;\forall j\right\}. 
\end{align}
We define $\varepsilon_j\in \mathfrak t^*$ by $\varepsilon_j(X)=\theta_j$ for $X$ in $\mathfrak t$ as above, for $1\leq j\leq n+2$. 
Then, $\Phi(\mathfrak g,\mathfrak t)= \{\pm \varepsilon_i\pm\varepsilon_j: 1\leq i<j\leq n+2\}$.

The maximal torus $T\cap K$ of $K$ satisfy $(\mathfrak k\cap \mathfrak t)^* = \op{span}_\C \{\varepsilon_1,\dots,\varepsilon_{n}\}$, thus $\Phi(\mathfrak k,\mathfrak t)= \{\pm \varepsilon_i\pm\varepsilon_j: 1\leq i<j\leq n\}\cup \{\pm\varepsilon_i: 1\leq i\leq n\}$. 
Similarly, the maximal torus $T\cap H$ in $H$ satisfies that $(\mathfrak h\cap \mathfrak t)^*=\op{span}_\C \{\varepsilon_{n+2}\}$ and $\Phi(\mathfrak h,\mathfrak t)= \{\pm \varepsilon_{n+2}\}$.

We pick compatible orders in $\mathfrak t^*$, $(\mathfrak k\cap \mathfrak t)^*$ and $(\mathfrak h\cap \mathfrak t)^*$, determined by the lexicographic order with respect to the ordered basis $\{\varepsilon_1, \dots,\varepsilon_{n+2}\}$. 
Thus
\begin{align}
\label{eq2n+1:positiverootsG}
\Phi^+(\mathfrak g,\mathfrak t) &= \{ \varepsilon_i\pm\varepsilon_j: 1\leq i<j\leq n+2\}, \qquad
\Phi^+(\mathfrak h,\mathfrak t) = \{\varepsilon_{n+2}\},\\
\label{eq2n+1:positiverootsKH}
\Phi^+(\mathfrak k,\mathfrak t) &= \{ \varepsilon_i\pm\varepsilon_j: 1\leq i<j\leq n\}\cup \{\varepsilon_j: 1\leq j\leq n\},
\end{align}
\begin{align}
P(G) &= \oplus_{j=1}^{n+2} \Z\varepsilon_j,&
	P^{++}(G) &= \{\textstyle\sum_{j=1}^{n+2} \lambda_j\varepsilon_j\in P(G): \lambda_1\geq \dots\geq \lambda_{n+1}\geq |\lambda_{n+2}|\}, \\
P(K) &= \oplus_{j=1}^{n} \Z\varepsilon_j,&
	P^{++}(K) &= \{\textstyle\sum_{j=1}^{n} \lambda_j\varepsilon_j\in P(K): \lambda_1\geq \dots\geq \lambda_{n}\geq 0\},\\
P(H) &=  \Z\varepsilon_{n+1},&
	P^{++}(H) &= \{k\varepsilon_{n+2}\in P(H): k\geq0\},
\end{align}
\begin{align}
	\rho_{\mathfrak g} &:= \sum_{i=1}^{n+2} (n+2-i)\varepsilon_i, &
	\rho_{\mathfrak k} &:= \sum_{i=1}^{n} (n+\tfrac12-i)\varepsilon_i, &
	\rho_{\mathfrak h} &:= \tfrac12 \varepsilon_{n+2}.
\end{align}
For $\lambda\in P^{++}(G)$, $\mu\in P^{++}(K)$, and $k\varepsilon_{n+2}\in P^{++}(H)$, we denote by $\pi_\lambda, \sigma_\mu, \tau_{k}$ the corresponding irreducible representations of $G$, $K$ and $H$ respectively.

The Weyl group $W_{\mathfrak g}$ consists in elements $\omega=s p$, with $p$ a permutation of the $n+2$ coordinates and $s$ a multiplication by $-1$ on a subset of coordinates with even cardinality. 
We still denote by $s_i:\mathfrak t^*\to \mathfrak t^*$ the reflexion with respect to the axis $i$ like in Subsection~\ref{subsec2n:notation}.
Furthermore, we define $p_{i,j}$ to be the transposition of the coordinates $i$ and $j$. 

We consider the inner product $\langle\cdot,\cdot\rangle$ on $\mathfrak g$ given by $\langle X,Y\rangle = \tfrac12 \tr(XY)$. 
We extend it to $\mathfrak t^*$.
It turns out that $\{\varepsilon_1,\dots,\varepsilon_{n+2}\}$ is an orthonormal basis of $\mathfrak t^*$.

\subsection{Main theorem for type D case}
The main result in this section is the following.

\begin{theorem}\label{thm2n+1:mainthm} 
Let $n\geq1$, $G=\SO(2n+4)$, $K=\SO(2n+1)$, $H=\SO(3)$, $\lambda=\sum_{i=1}^{n+2}\lambda_i \varepsilon_i\in P^{++}(G)$, $\mu=\sum_{i=1}^n\mu_i \varepsilon_i\in P^{++}(K)$.
If $\mu$ simply interlaces $\lambda$, i.e.\ $\lambda_i\geq \mu_i\geq \lambda_{i+1}$ for $1\leq i\leq n$, then
\begin{equation}\label{eq2n+1:tensores}
\Hom_K(\sigma_\mu,\pi_\lambda) \simeq \left( \bigoplus_{k=|\lambda_{n+2}|}^{\lambda_{n+1}} \tau_{k} \right) \otimes \bigotimes_{m=1}^n \left( \bigoplus_{j=0}^{\lfloor (\lambda_m-\mu_m)/2\rfloor }  \tau_{\lambda_m-\mu_m-2j} \right)
\end{equation}
as $H$-modules. 
\end{theorem}

Similarly as in the previous section, we first make the first steps of applying Kostant's Branching Formula \eqref{eq:Kostant} to $(G, K\times H)$.
The maximal torus $S:=T\cap (K\times H)$ in $K\times H$ misses the $(n+1)$-th $2\times 2$-block in \eqref{eq2n+1:toromaximal}. 
Thus, $\mathfrak s^*=\op{span}_\C \{\varepsilon_1,\dots,\varepsilon_n,\varepsilon_{n+2}\}$ and the restriction from $\mathfrak t^*$ to $\mathfrak s^*$ denoted by a bar removes the $(n+1)$-th coordinate, that is, 
\begin{equation}
\textstyle\overline{\sum\limits_{i=1}^{n+2} \beta_i\varepsilon_i }  =  \sum\limits_{i=1}^{n} \beta_i\varepsilon_i  + \beta_{n+2}\varepsilon_{n+2}. 
\end{equation}
Furthermore, $\Phi^+(\mathfrak k\times\mathfrak h,\mathfrak s)=\Phi^+(\mathfrak k,\mathfrak t\cap\mathfrak k)\cup \Phi^+(\mathfrak h,\mathfrak t\cap\mathfrak h)$, thus
\begin{equation}
\Sigma=  \{\varepsilon_i\pm \varepsilon_{n+2}: 1\leq i\leq n\} \cup \{\varepsilon_i: 1\leq i\leq n\}\cup \{-\varepsilon_{n+2} \},
\end{equation}
each element with multiplicity one. 
It follows that, for $\nu \in\mathfrak s^*$, 
\begin{equation}\label{eq2n+1:P_Sigma}
	\mathcal P_\Sigma(\nu)>0 
	\quad \Longrightarrow \quad
	\nu_i:= \langle\varepsilon_i,\nu \rangle \in \Z \quad\forall i,\; \nu_i\geq0 \quad\forall \,1\leq i\leq n, \; \nu_{n+2}\leq \textstyle \sum\limits_{i=1}^n \nu_i. 
\end{equation}

From \eqref{eq:Kostant}, we obtain that the number of times that $\sigma_\mu\otimes \tau_k$ appears in $\pi_\lambda|_{K\times H}$ is given by 
\begin{equation}\label{eq2n+1:Kostant}
\dim \Hom_{K\times H}({\sigma_\mu}\otimes {\tau_k},{\pi_\lambda})
= \sum_{\omega \in W_{\mathfrak g}} \op{sgn}(\omega) \;  \mathcal P_{\Sigma} \big( \, \overline{ \omega(\lambda+\rho_{\mathfrak g})- \rho_{\mathfrak g}} -\mu- k\varepsilon_{n+2}\big).
\end{equation}
The next lemmas indicate the non-zero terms in the above sum. 
The proof of the first one is completely analogous to the proof of Lemma~\ref{lem2n:Weyl1}.

\begin{lemma}\label{lem2n+1:Weyl1}
Assume $\lambda_{n+2}\geq0$. If $\omega \in W_{\mathfrak g}$ satisfies that the $\omega$-th term in \eqref{eq2n+1:Kostant} is non-zero, then $\omega=p $ or $\omega= s_{n+1}s_{n+2}p$, for some permutation $p$. 
\end{lemma}

\begin{lemma}\label{lem2n+1:Weyl2}
Assume $\lambda_i\geq \mu_i\geq \lambda_{i+1}$ for all $1\leq i\leq n$. If $\omega \in W_{\mathfrak g}$ satisfies that the $\omega$-th term in \eqref{eq2n+1:Kostant} is non-zero, then $\omega$ is in $\{1,\; s_{n+1}\,s_{n+2},\; p_{n+1,n+2},\; s_{n+1}\,s_{n+2}\, p_{n+1,n+2}\} $. 
\end{lemma}

\begin{proof}
We write $\omega = sp$ with $s=1$ or $s= s_{n+1}s_{n+2}$ by Lemma~\ref{lem2n+1:Weyl1}.
By proceeding as in the proof of Theorem \ref{thm2n:mainthm} we obtain that $\omega(\varepsilon_i)=\varepsilon_i$ for all $1\leq i\leq n-1$.

Suposse $\omega(\varepsilon_{n+1})=\varepsilon_n$, then $0\leq \langle \varepsilon_n, \omega(\lambda+\rho_{\mathfrak g})-\rho_{\mathfrak g}-\mu-k\varepsilon_{n+2}\rangle =  \lambda_{n+1}-\mu_n-1<0$ which is a contradiction. Finally, if  $\omega(\varepsilon_{n+2})=\varepsilon_n$, then $0\leq \langle \varepsilon_n, \omega(\lambda+\rho_{\mathfrak g})-\rho_{\mathfrak g}-\mu-k\varepsilon_{n+2}\rangle =  \lambda_{n+2}-\mu_n-2<0$ which is again a contradiction. Thus, $\omega(\varepsilon_{n})=\varepsilon_n$ and the claim follows.
\end{proof}

We now show that it is sufficient to prove Theorem~\ref{thm2n+1:mainthm} for $\lambda_{n+2}\geq0$. 
We set $\widetilde\lambda=\sum_{i=1}^{n+1} \lambda_i\varepsilon_i- \lambda_{n+2}\varepsilon_{n+2}$ for any $\lambda=\sum_{i=1}^{n+2} \lambda_i\varepsilon_i \in P^{++}(G)$. 
Note $\widetilde \lambda \in P^{++}(G)$.

\begin{lemma}\label{lem2n+1:mu_n>0}
For any $\lambda\in P^{++}(G)$, $\mu\in P^{++}(K)$, $k\geq0$, we have that $\pi_{\lambda}|_{K\times H} \simeq \pi_{\widetilde \lambda}|_{K\times H}$. 
\end{lemma}

\begin{proof}
We set $g_0\in \diag(1,\dots,1,-1) \in \Ot(2n+4)$. 
Although $g_0$ is not in $G$, the map $\varphi: x \mapsto g_0xg_0$ is an automorphism of $G$. 
It turns out that $\pi_\lambda\circ\varphi \simeq\pi_{\widetilde \lambda}$, $\sigma_\mu\circ\varphi|_K =\sigma_{\mu}$, and $\tau_k\circ\varphi|_H \simeq \tau_k$.
The assertions then follows. 
\end{proof}

We are now in position to prove the main theorem of this section.

\begin{proof}[Proof of Theorem~\ref{thm2n+1:mainthm}]
The strategy is the same as in the proof of Theorem~\ref{thm2n:mainthm}. 
By Lemma~\ref{lem2n+1:mu_n>0}, we may assume $\lambda_{n+2}\geq0$.

From \eqref{eq2n+1:Kostant} and Lemmas~\ref{lem2n+1:Weyl1} and \ref{lem2n+1:Weyl2}, we obtain that 
\begin{multline} \label{eq2n+1:taulLHS}
\dim \Hom_{K\times H}({\sigma_\mu}\otimes {\tau_k}, {\pi_\lambda}) =
\mathcal P_{\Sigma}(\bar\lambda -\mu -k\varepsilon_{n+2}) 
+ \mathcal P_{\Sigma}(\bar \lambda -\mu -(2\lambda_{n+2}+k)\varepsilon_{n+2})
\\ - \mathcal P_{\Sigma}(\bar\lambda -\mu +(\lambda_{n+1}-\lambda_{n+2}+1-k)\varepsilon_{n+2}) 
- \mathcal P_{\Sigma}(\bar\lambda - \mu -(\lambda_{n+1}+\lambda_{n+2}+1+k)\varepsilon_{n+2}). 
\end{multline}
Set $\Sigma''=\Sigma'\cup \{-\varepsilon_{n+2}\}$, where $\Sigma'=\{\varepsilon_i\pm \varepsilon_{n+2}: 1\leq i\leq n\}$. 
One has that, see for instance \cite[(9.56)]{Knapp-book-beyond},
$ \mathcal P_{\Sigma''}(\nu) = \mathcal P_{\Sigma''}(\nu+m\varepsilon_{n+2}) + \sum_{r=0}^{m-1}\mathcal P_{\Sigma'}(\nu+r\varepsilon_{n+2}) $ 
for all $\nu\in \mathfrak s^*$. 
Then, 
\begin{multline}\label{eq2n+1:recursionP1}
\mathcal P_{\Sigma}(\bar\lambda -\mu -k\varepsilon_{n+2})- \mathcal P_{\Sigma}(\bar\lambda -\mu +(\lambda_{n+1}-\lambda_{n+2}+1-k)\varepsilon_{n+2})  \\ 
= \sum_{\gamma_1=0}^{\lfloor \frac{\lambda_1-\mu_1}{2}\rfloor} \dots \sum_{\gamma_n=0}^{\lfloor \frac{\lambda_n-\mu_n}{2}\rfloor} \sum_{I\subset \intervalo} \Big(
\mathcal P_{\Sigma''}\big(\bar\lambda - \mu -k\varepsilon_{n+2}-2\gamma - \beta_I \big) \\  \shoveright{ - \mathcal P_{\Sigma''}\big(\bar\lambda - \mu +(\lambda_{n+1}-\lambda_{n+2}+1-k)\varepsilon_{n+2}-2\gamma - \beta_I \big) \Big)}\\
= \sum_{r=\lambda_{n+2}}^{\lambda_{n+1}}
	\sum_{\gamma_1=0}^{\lfloor \frac{\lambda_1-\mu_1}{2}\rfloor} \dots \sum_{\gamma_n=0}^{\lfloor \frac{\lambda_n-\mu_n}{2}\rfloor} \sum_{I\subset \intervalo}
\mathcal P_{\Sigma'}\big(\bar\lambda - \mu +(r-\lambda_{n+2}-k)\varepsilon_{n+2}-2\gamma - \beta_I \big) \\
= \sum_{r=\lambda_{n+2}}^{\lambda_{n+1}} \sum_{\gamma_1=0}^{\lfloor \frac{\lambda_1-\mu_1}{2}\rfloor} \dots \sum_{\gamma_n=0}^{\lfloor \frac{\lambda_n-\mu_n}{2}\rfloor} 
\mathcal P_{\Sigma'}\big(2(\bar\lambda - \mu +(r-\lambda_{n+2}-k)\varepsilon_{n+2}-2\gamma )\big) .
\end{multline}
The last equality follows by Lemma~\ref{lem:P_Sigma'}. 
Similarly, one obtains that 
\begin{multline}\label{eq2n+1:recursionP2}
\mathcal P_{\Sigma}(\bar \lambda -\mu -(2\lambda_{n+2}+k)\varepsilon_{n+2})
- \mathcal P_{\Sigma}(\bar\lambda - \mu -(\lambda_{n+2}+\lambda_{n+1}+1+k)\varepsilon_{n+2}) \\ 
= - \sum_{r=\lambda_{n+2}}^{\lambda_{n+1}} \sum_{\gamma_1=0}^{\lfloor \frac{\lambda_1-\mu_1}{2}\rfloor} \dots \sum_{\gamma_n=0}^{\lfloor \frac{\lambda_n-\mu_n}{2}\rfloor} 
\mathcal P_{\Sigma'}\big(2(\bar\lambda - \mu -(r+\lambda_{n+2}+1+k)\varepsilon_{n+2}-2\gamma)\big). 
\end{multline}

Now, substituting \eqref{eq2n+1:recursionP1} and \eqref{eq2n+1:recursionP2} in \eqref{eq2n+1:taulLHS}, we get
\begin{multline} \label{eq2n+1:taulLHS2}
\dim \Hom_H({\tau_k}, \Hom_{K}({\sigma_\mu}, {\pi_\lambda})) =
\dim \Hom_{K\times H}({\sigma_\mu}\otimes {\tau_k}, {\pi_\lambda}) \\
=\sum_{r=\lambda_{n+2}}^{\lambda_{n+1}} \sum_{\gamma_1=0}^{\lfloor \frac{\lambda_1-\mu_1}{2}\rfloor} \dots \sum_{\gamma_n=0}^{\lfloor \frac{\lambda_n-\mu_n}{2}\rfloor} 
\left(\mathcal P_{\Sigma'}\big(2(\bar\lambda - \mu +(r-\lambda_{n+2}-k)\varepsilon_{n+2}-2\gamma )\big) \right.\\\left.- \mathcal P_{\Sigma'}\big(2(\bar\lambda - \mu -(r+\lambda_{n+2}+1+k)\varepsilon_{n+2}-2\gamma)\big) \right). 
\end{multline}

Again, the right-hand side of \eqref{eq2n+1:tensores} is
\begin{equation}\label{eq2n+1:RHS}
\bigoplus_{r=\lambda_{n+2}}^{\lambda_{n+1}}\bigoplus_{\gamma_1=0}^{\lfloor \frac{\lambda_1-\mu_1}{2}\rfloor} 
\dots \bigoplus_{\gamma_n=0}^{\lfloor \frac{\lambda_n-\mu_n}{2}\rfloor}
\tau_{\lambda_1-\mu_1-2\gamma_1}\otimes\dots \otimes \tau_{\lambda_n-\mu_n-2\gamma_n} \otimes \tau_{r}.
\end{equation}
By Theorem~\ref{thm:ClebschGordan}, $\tau_k$ appears in the factor $\tau_{\lambda_1-\mu_1-2\gamma_1}\otimes\dots \otimes \tau_{\lambda_n-\mu_n-2\gamma_n} \otimes \tau_{r}$ with coefficient
\begin{equation*}
\textstyle \mathcal P_{\Sigma'}\left( 2\big(\bar\lambda-\mu-2\gamma \big) +2(r-\lambda_{n+2}-k)\varepsilon_{n+2}\right) 
- \mathcal P_{\Sigma'}\left( 2\big(\bar\lambda-\mu -2\gamma\big)+ 2(r-\lambda_{n+2}+k+1)\varepsilon_{n+2}\right).
\end{equation*}
Since $\mathcal P_{\Sigma'}(\alpha)=\mathcal P_{\Sigma'}(s_{n+2}\alpha)$, it follows that the number of times that $\tau_k$ occurs in \eqref{eq2n+1:RHS} coincides with \eqref{eq2n+1:taulLHS2}, as asserted. 
\end{proof}

\subsection{Prescribed highest weight end for type D case}
We next show the analogous result to Theorem~\ref{thm2n:ending} for the type D case, which gives an explicit decomposition of the multiplicity space as an $H$-module when the ending coefficients of $\lambda$ coincide with a part of the coefficients of $\mu$. 
This result will be also proved by Tsukamoto's branching law from $G$ to $K\times H$.

\begin{theorem}\label{thm2n+1:ending}
Let $G=\SO(2n+4)$, $K=\SO(2n+1)$, and $H=\SO(3)$ for any $n\geq1$. 
For $\mu =\sum_{i=1}^{n}\mu_i \varepsilon_i\in P^{++}(K)$ and $\lambda=\sum_{i=1}^{n+2} \lambda_i\varepsilon_i \in P^{++}(G)$ with $|\lambda_{i+3}|=\mu_i$ for all $1\leq i\leq n-1$ and $\mu_n\leq |\lambda_{n+2}|$, then
\begin{equation}
\Hom_K(\sigma_\mu,\pi_\lambda)\simeq 	\pi'_{\lambda'}|_H\otimes \tau_{\mu_n},
\end{equation}
as $H$-modules, where $\pi'_{\lambda'}$ denotes the irreducible representation of $\U(3)$ with highest weight $\lambda': = \lambda_1\varepsilon_1'+\lambda_2\varepsilon_2'+ |\lambda_3|\varepsilon_3'$ (see Subsection~\ref{subsec:U(3)}). 
\end{theorem}

We note that the condition $\mu_n\leq |\lambda_{n+2}|$ is redundant unless $n=1$. Furthermore, $|\lambda_{i+3}|=\lambda_{i+3}$ for all $i\leq n-2$. 

Fix $\lambda\in P^{++}(G)$ with $\lambda_{n+2}\geq 0$. 
It will be convenient to set $\lambda_{n+3}=\mu_{n+1}=0$ and $\mu_0=\mu_{-1}=\lambda_1$.
We recall from \eqref{eq:characterrestriction}, that the character of $\pi|_{K\times H}$ is given by $\overline{\chi_\pi}$, for any finite dimensional representation $\pi$ of $G$. 
Tsukamoto established in the proof of Theorem~4 in \cite{Tsukamoto05} that the character of the restriction of $\pi_\lambda$ to $K\times H$ is given by
\begin{equation}\label{eq2n+1:Tsukamotocharacters}
\chi_{\pi_\lambda|_{K\times H}} =  \overline{\chi_{\pi_\lambda}} = 
\frac{1}{(e^{\frac12{\varepsilon_{n+1}}}-  e^{-\frac12{\varepsilon_{n+1}}})} 
\sum_{\mu} \;\chi_{ \sigma_\mu}\;
\sum_{(a_1,\dots,a_{n+1})} \frac{\prod\limits_{i=1}^{n+1} (e^{l_i\varepsilon_{n+2}}-e^{-l_i\varepsilon_{n+2}})} {(e^{\varepsilon_{n+2}}-  e^{-\varepsilon_{n+2}})^n  } ,
\end{equation}
where the first sum is over every $\mu\in P^{++}(K)$ \emph{triply interlacing} $\lambda$, that is, $\lambda_i\geq \mu_i\geq \lambda_{i+3}$ for all $1\leq i\leq n$, the second sum is over the $(n+1)$-tuples $(a_1,\dots,a_{n+1})\in\Z^{n+1}$ satisfying that $a_1 \geq \dots \geq  a_{n+1}\geq0$ and $\max(\mu_i,\lambda_{i+1}) \leq a_i\leq \min(\mu_{i-2},\lambda_i)$ for all $1\leq i\leq n+1$, and the parameters $l_1,\dots,l_{n+1}$ are given by
\begin{equation}\label{eq2n+1:l_i}
\begin{cases}
l_i = \min(\mu_{i-1},a_{i}) - \max(\mu_{i}, a_{i+1})+1 & \quad\text{for all $1\leq i\leq n$,}\\
l_{n+1} = \min(\mu_{n},a_{n+1}) +1/2.
\end{cases}
\end{equation}
He thus got the next implicit branching law from $G=\SO(2n+4)$ to $K\times H=\SO(2n+1)\times\SO(3)$.

\begin{theorem}\cite[Theorem~4]{Tsukamoto05}\label{thm2n+1:Tsukamoto}
Let $\lambda\in P^{++}(G)$, $\mu\in P^{++}(K)$, and $k\varepsilon_{n+1}\in P^{++}(H)$. 
If $\lambda_i\geq \mu_i\geq \lambda_{i+3}$ for all $1\leq i\leq n-2$, $\lambda_{n-1}\geq \mu_{n-1}\geq |\lambda_{n+2}|$, and $\lambda_{n}\geq \mu_{n}$, then the number of times that $\sigma_\mu\otimes \tau_k$ occurs in $\pi_\lambda|_{K\times H}$ is given by $m_{k}$, where the coefficients $m_p$ for $p\geq0$ are defined by
\begin{equation}\label{eq2n+1:Tsukamoto}
\sum_{(a_1,\dots,a_n)} \frac{\prod\limits_{i=1}^{n+1} (e^{l_i\varepsilon_{n+2}}-e^{-l_i\varepsilon_{n+2}})} {(e^{\varepsilon_{n+2}}-  e^{-\varepsilon_{n+2}})^n}  = \sum_{p\geq0} m_p \; \big(e^{(p+\frac12)\varepsilon_{n+2}}- e^{-(p+\frac12)\varepsilon_{n+2}}\big),
\end{equation}
where the sum at the left is over the $(n+1)$-tuples $(a_1,\dots,a_{n+1})\in\Z^{n+1}$ satisfying that $a_1 \geq \dots \geq  a_{n+1}\geq0$, $\max(\mu_i,\lambda_{i+1}) \leq a_i\leq \min(\mu_{i-2},\lambda_i)$ for all $1\leq i\leq n$, $|\lambda_{n+2}| \leq a_{n+1}\leq \min(\mu_{n-1},\lambda_{n+1})$, and $l_1,\dots,l_{n+1}$ are given by \eqref{eq2n+1:l_i}. 
Otherwise, $\sigma_\mu\otimes \tau_k$ does not occur in $\pi_\lambda|_{K\times H}$. 
\end{theorem}

\begin{proof}[Proof of Theorem~\ref{thm2n+1:ending}]
From Lemma~\ref{lem2n+1:mu_n>0}, we assume that $\lambda_{n+2}\geq0$.
We will first show that 
\begin{equation}\label{eq2n+1:reduccion}
\Hom_K(\sigma_\mu,\pi_\lambda)\simeq \Hom_K(\sigma_0,\pi_{\lambda'})\otimes \tau_{\mu_n}
\end{equation}
as $H$-modules, where $\lambda'=\lambda_1\varepsilon_1+ \lambda_2\varepsilon_2 + \lambda_3\varepsilon_3$. 
We need to check that the term accompanying $\chi_{\sigma_\mu}$ in $\chi_{\pi_\lambda|_{K\times H}}$ coincides with the term accompanying $\chi_{\sigma_0}$ in $\chi_{\pi_{\lambda'}|_{K\times H}}$ times $\chi_{\tau_{\mu_n}}$.

By assumption, $\lambda_{i+3} = \mu_i$ for all $1\leq i\leq n-1$.
Then, \eqref{eq2n+1:Tsukamotocharacters} yields that the term accompanying $\chi_{\sigma_\mu}$ in $\chi_{\pi_\lambda|_{K\times H}}$ is equal to 
\begin{equation}
\chi_{\tau_{\mu_n}}
\sum_{a_1= \lambda_2}^{\lambda_1} \sum_{a_2=\lambda_3}^{\lambda_2}\,  \frac{e^{(a_1-a_2+1)\varepsilon_{n+1}}-e^{-(a_1-a_2+1)\varepsilon_{n+1}} } {e^{\varepsilon_{n+1}}-  e^{-\varepsilon_{n+1}}},
\end{equation}
On the other hand, also by \eqref{eq2n+1:Tsukamotocharacters}, the term accompanying $\chi_{\sigma_0}$ in $\chi_{\pi_{\lambda'}|_{K\times H}}$ is given by 
\begin{equation}
\sum_{a_1= \lambda_2}^{\lambda_1} \sum_{a_2=\lambda_3}^{\lambda_2}\,  \frac{e^{(a_1-a_2+1)\varepsilon_{n+1}}-e^{-(a_1-a_2+1)\varepsilon_{n+1}} } {e^{\varepsilon_{n+1}}-  e^{-\varepsilon_{n+1}}},
\end{equation}
which completes the proof of \eqref{eq2n+1:reduccion}. 

The proof follows by applying Knapp's duality \cite{Knapp01} to the right-hand side of \eqref{eq2n+1:reduccion}. 
\end{proof}

\begin{remark}
Theorem~\ref{thm2n+1:ending} and \eqref{eq:U(3)toSO(3)}, give an explicit expression for the number of times that $\sigma_\mu\otimes \tau_k$ occurs in $\pi_{\lambda}|_{K\times H}$ for $\lambda$ and $\mu$ as in Theorem~\ref{thm2n+1:ending}. 
\end{remark}

\bibliographystyle{plain}

\end{document}